\documentclass[12pt,reqno,a4wide]{amsart}
\usepackage{amsfonts}
\usepackage{hyperref}
\usepackage{tikz} 
\tikzset{
	arrow/.pic={\path[tips,every arrow/.try,->,>=#1] (0,0) -- +(.1pt,0);},
	pics/arrow/.default={triangle 45}
}
\usetikzlibrary{decorations.pathreplacing}
\usetikzlibrary{arrows,decorations.pathmorphing,snakes}
\tikzset{snake it/.style={decorate, decoration={snake, amplitude = 2, segment length = 5,}}}

\theoremstyle{plain}
\newtheorem{theorem}{Theorem}

\newtheorem{example}{Example}

\theoremstyle{definition}

\usetikzlibrary{decorations.pathreplacing}
\usetikzlibrary{arrows,decorations.pathmorphing,snakes}
\usetikzlibrary{fit}

\allowdisplaybreaks

\begin{document}
	
\title{On the Mellin transform of $\frac{\log^{n}(1+x)}{(1+x)^{m+1}}$}

\author[Sumit Kumar Jha]{Sumit Kumar Jha}
\address{International Institute of Information Technology\\
Hyderabad, India}
\email{kumarjha.sumit@research.iiit.ac.in}


\begin{abstract}
We use the Ramanujan's master theorem to evaluate the integral $$\int_{0}^{\infty}\frac{x^{l-1}}{(1+x)^{m+1}}\log^{n}(1+x)\, dx $$ in terms of the digamma function, the gamma function, and the Hurwitz zeta function. 
\end{abstract}
\maketitle
\section{Main Result}
The book \cite{Mellin} has no entry for the Mellin transform of the function 
$$
\frac{\log^{n}(1+x)}{(1+x)^{m+1}}.
$$
We note in the following result that the Mellin transform of the above function can be evaluated when $n$ is a non-negative integer using a series expansion.
\begin{theorem}
\label{main}
For all non-negative integers $n$ we have
\begin{align*}
\int_{0}^{\infty}\frac{x^{l-1}}{(1+x)^{m+1}}\log^{n}(1+x)\, dx
=(-1)^{n}P_{n}(\psi(m+1-l)-\psi(m+1), \zeta(2,m+1)-\zeta(2,m+1-l),
\\ \zeta(3,m+1)-\zeta(3,m+1-l), \cdots, \zeta(n,m+1)-\zeta(n,m+1-l))\cdot 
\frac{\Gamma(m+1-l)\cdot \Gamma(l)}{\Gamma(m+1)}
\end{align*}
where $m$ and $l$ are complex numbers for which the integral on the left side converges, $\Gamma(s)$ is the gamma function, $\psi(s)=\frac{\Gamma'(s)}{\Gamma(s)}$ is the digamma function, $\zeta(s,n)$ is the Hurwitz zeta function, and the polynomial $P_{n}(s_{1},\cdots,s_{n})$ is defined by $P_{0}=1$ and
$$
P_{n}(s_{1},\cdots,s_{n})=(-1)^{n}Y_{n}(-s_{1},-s_{2},-2s_{3},\cdots,-(n-1)!s_{n})
$$
where $Y_{n}$ is the familiar Bell polynomial. The first few values of the polynomials being:
$$ P_{1}(s_{1})=s_{1},$$
$$ P_{2}(s_{1},s_{2})=s_{1}^{2}-s_{2}, $$
$$ P_{3}(s_{1},s_{2},s_{3})=s_{1}^{3}-3s_{1}s_{2}+2s_{3}, $$
$$ P_{4}(s_{1},s_{2},s_{3},s_{4})=s_{1}^{4}-6s_{1}^{2}s_{2}+8s_{1}s_{3}+3s_{2}^{2}-6s_{4}. $$
\end{theorem}
\begin{proof}
Zave \cite{Zave} proved the following series expansion:
\begin{equation}
\frac{\log^{n}(1+x)}{(1+x)^{m+1}}=(-1)^{n}\sum_{k=0}^{\infty}P_{n}(H_{m+k}^{(1)}-H_{m}^{(1)},H_{m+k}^{(2)}-H_{m}^{(2)},\cdots, H_{m+k}^{(n)}-H_{m}^{(n)})\binom{m+k}{m}(-x)^{k}
\end{equation}
where
$$
H_{n}^{(k)}=1+\frac{1}{2^{k}}+\frac{1}{3^k}+\cdots+\frac{1}{n^k}.
$$
We first recall that
$$ H_{n}^{(1)}=\psi(n+1)+\gamma $$
$\gamma$ being the Euler's constant, and
$$
H_{n}^{(s)}=\zeta(s)-\zeta(s,n+1),
$$
for $s\geq 2$.\par 
Now recalling the Ramanujan's master theorem \cite{Hardy} which states that
\begin{equation*}
    \int_{0}^{\infty} x^{l-1}\{ \phi(0)-x \phi(1)+x^{2} \phi(2)-\cdots\}\, dx=\frac{\pi}{\sin{l \pi}}\phi(-l)
\end{equation*}
gives us our result whenever $m$ and $l$ are chosen such that the integral is convergent.
\end{proof}
\section{Examples}
\begin{example}
Substituting $n=1$ in our result Theorem \ref{main} gives us:
$$
\int_{0}^{\infty}\frac{x^{l-1}}{(1+x)^{m+1}}\log(1+x)\, dx
=(\psi(m+1)-\psi(m+1-l))\cdot \frac{\Gamma(m+1-l)\cdot \Gamma(l)}{\Gamma(m+1)}.
$$
\end{example}
\begin{example}
By letting $m\rightarrow 0$ and $l\rightarrow 0$ in Theorem 1 we have
$$
\int_{0}^{\infty}\frac{x^{n}}{e^{x}-1}\,dx= \int_{0}^{\infty}\frac{\log^{n}(1+t)}{t(1+t)}\, dt=n!\, \zeta(n+1).
$$
\end{example}
\begin{example}
By letting $l\rightarrow 1/2$, $m\rightarrow 0$, and $n=2$ we get
$$
\int_{0}^{\infty}\frac{\log^{2}(1+x)}{\sqrt{x}\, (1+x)}\, dx=\frac{\pi^3}{3}+4\,\pi\, \log^{2}{2}.
$$
Mathematica is unable to arrive at this form. 
\end{example}
\begin{example}
By letting $l\rightarrow 3/4$, $m\rightarrow 0$, and $n=2$ we get
$$
\int_{0}^{\infty}\frac{\log^{2}(1+x)}{\sqrt[4]{x}\, (1+x)}\, dx=\sqrt{2}\, \pi \left(8\, C+\frac{5\,\pi^2}{6}+\left(\gamma+\psi\left(\frac{1}{4}\right)\right)^2\right),
$$
where $C=\sum_{n=0}^{\infty}\frac{(-1)^{n}}{(2n+1)^{2}}$ is the Catalan's constant, and $\gamma$ is the Euler's constant. Mathematica is unable to arrive at this form.
\end{example}

\end{document}